\documentclass[12pt]{amsart}

\usepackage[all]{xy}
\usepackage{fullpage}
\usepackage{latexsym}
\usepackage{amsmath}
\usepackage{amsfonts}
\usepackage{amssymb}
\usepackage{amsthm}
\usepackage{eucal}
\usepackage{enumerate,yfonts}
\usepackage{mathrsfs}
\usepackage{graphicx}
\usepackage{graphics}
\usepackage{epstopdf}
\usepackage{amscd}
\usepackage{bbm}
\usepackage{hyperref}
\usepackage{url}
\usepackage{color}
\usepackage{bbm}
\usepackage{cancel}
\usepackage{enumerate}
\usepackage{amsmath,amsthm}
\usepackage{amssymb}
\usepackage{epsfig}
\usepackage{pstricks}
\usepackage{xy}
\usepackage{xypic}
\usepackage{epigraph}

\newtheorem{thm}{Theorem}[section]
\newtheorem{corollary}[thm]{Corollary}
\newtheorem{lemma}[thm]{Lemma}

\newtheorem{thm-dfn}[thm]{Theorem-Definition}

\theoremstyle{definition}
\newtheorem{definition}[thm]{Definition}

\numberwithin{equation}{section}

\theoremstyle{remark}
\newtheorem{remark}{Remark}[section]

\newcommand{\on}{\operatorname}

\newcommand{\Loc}{\on{Loc}}

\newcommand{\quash}[1]{}  %%Anything in \quash is ignored
\newcommand{\nc}{\newcommand}

\newcommand{\calF}{{\mathcal F}}

\nc{\al}{{\alpha}} \nc{\be}{{\beta}} \nc{\ga}{{\gamma}}
\nc{\ve}{{\varepsilon}} \nc{\Ga}{{\Gamma}} %\nc{\la}{{\lambda}}
\nc{\La}{{\Lambda}}

\nc{\ad }{{\on{ad }}}

\nc{\aff}{{\on{aff}}} \nc{\Aff}{{\mathbf{Aff}}}

\nc{\der}{{\on{der}}}

\nc{\diag}{{\on{diag}}}

\nc{\Fl}{{\calF\ell}}

\nc{\Hg}{{\on{Higgs}}}
\newcommand{\Hom}{{\on{Hom}}}

\nc{\Id}{{\on{Id}}}

\nc{\Ind}{{\on{Ind}}}

\nc{\Op}{{\on{Op}}}

\nc{\res}{{\on{res}}}

\nc{\tr}{{\on{tr}}}

\nc{\GSp}{{\on{GSp}}} \nc{\GU}{{\on{GU}}} \nc{\SL}{{\on{SL}}}
\nc{\SU}{{\on{SU}}} \nc{\SO}{{\on{SO}}}

\nc{\nh}{{\Loc_{J^p}(\tau')}}
\nc{\bnh}{{\Loc_{\breve J^p}(\tau')}}

\nc{\bU}{{\overline{U}}} \nc{\IC}{{\on{IC}}}

\newcommand{\beqn}{\begin{equation*}}
\newcommand{\eeqn}{\end{equation*}}

\newcommand{\beq}{\begin{equation}}
\newcommand{\eeq}{\end{equation}}

\newcommand{\Irr}{\operatorname{Irr}}

\newcommand{\Grs}{G^{\mathrm{rs}}}

\quash{

\setlength{\parskip}{2ex}
\setlength{\oddsidemargin}{0in}
\setlength{\evensidemargin}{0in}
\setlength{\textwidth}{6.5in}
\setlength{\topmargin}{-0.15in}
\setlength{\textheight}{8.6in}

\topmargin-0.5cm \textheight22cm \oddsidemargin1.2cm \textwidth14cm}
\begin{document}
\title{Admissibility of Bernstein centers}

\author{Tsao-Hsien Chen}
        \address{School of Mathematics, University of Minnesota, Twin cities, Minneapolis, MN 55455 }
         \email{chenth@umn.edu}
         \author{Cheng-Chiang Tsai}
        \address{Institute of Mathematics, Academia Sinica, 6F, Astronomy-Mathematics Building,
No. 1, Sec. 4, Roosevelt Road, Taipei, Taiwan\vskip.2cm
also National Sun Yat-Sen University and National Taiwan University}
         \email{chchtsai@as.edu.tw}
\thanks{}
\thanks{}

\maketitle     
\begin{abstract}
We provide a criterion for determining when elements of the Bernstein center of a totally disconnected locally compact group are admissible invariant distributions in the sense of Harish-Chandra \cite{HC99}. As a consequence, we deduce the local integrability results for elements of bounded depth or Bernstein supports in the Bernstein centers of reductive $p$-adic groups. Our methods apply uniformly to both complex and mod-$\ell$ coefficients, generalizing results of Moy and Tadi\'{c} \cite{MT02} in the complex case.

\end{abstract}
\section{Bernstein centers}
We recall some basic facts about the Bernstein center, following \cite{BD84}. Let $G$ be a totally disconnected locally compact group. Fix a prime $p$, and assume that $G$ contains a compact open pro-$p$ subgroup. Let $k$ be an algebraically closed field of characteristic $\ell\neq p$. Denote by $(H(G),*)$ the $k$-algebra of locally constant $k$-valued measures on $G$ with compact support. Let $R(G)$ be the category of smooth $k$-representations of $G$. Let $\on{Irr}(G)$ be the set of irreducible $k$-representations of $G$.

 There is an equivalence between 
$R(G)$ and the category of non-degenerate $H(G)$-modules, that is,
 $H(G)$-modules $V$ satisfying $H(G)\cdot V=V$.
The action of $f\in H(G)$ on a representation 
$(\pi,V_\pi)\in R(G)$
is given by
\[\pi(f): V_\pi\to V_\pi,\ \pi(f)(v)=\int_G f(g)\pi(g)(v).\]
Let \[Z(G)=\on{End}_{R(G)}(\on{Id})\] denote the Bernstein center of $G$.
Let $D(G)$ be the space of $k$-valued distributions on $G$, and let $D(G)^G$ denote the subspace of invariant distributions.
There is a canonical embedding \[\iota:Z(G)\to D(G)^G\]
given by $\iota(z)(h)=z(h^*)(\operatorname{id})$ where $h^*(g)=h(g^{-1})$ and $z(h^*)$ is the action of $z$ on $H(G)$ viewed as an $H(G)$-module.
The image of $\iota$ is the space of essentially compact invariant distributions 
\[D(G)^G_{ec}=\{f\in D(G)^G\;|\;f*h\in H(G)\text{\ \ for any\ \ } h\in H(G) \}.\]
For any $z\in Z(G)$,  we write
\[f_z=\iota(z)\in D(G)^G_{ec}\] the corresponding essentially compact invariant distribution.

Given a decomposition $R(G)=R(G)^{\heartsuit}\oplus R(G)^{\spadesuit}$, we obtain a decomposition
\[Z(G)=Z(G)^{\heartsuit}\oplus Z(G)^\spadesuit\] 
where $Z(G)^{\heartsuit}=\{z\in Z(G)|z|_{R(G)^{\spadesuit}}=0 \}$
and $Z(G)^{\spadesuit}=\{z\in Z(G)|z|_{R(G)^{\heartsuit}}=0 \}$.

\section{Admissibility}
We recall the notion of an admissible invariant distribution introduced by Harish-Chandra \cite[Definition 15.1]{HC99}:

\begin{definition}

\begin{enumerate}
    \item 
     An invariant distribution $f\in D(G)^G$ 
     is called 
    \emph{admissible} if for any point $x\in G$, there exists a compact open pro-$p$ subgroup 
    $U_x$ such that for any compact open $U_x'\subset U_x$
 and $(\rho,V_\rho)\in \Irr(U_x')$, we have 
\[\left(f*\Theta_\rho\right)|_{xU_x}\neq0\Rightarrow \exists g\in G\text{ such that } (V_\rho)^{U_x'\cap gU_xg^{-1}}\neq 0.\]
    Here $\Theta_\rho\in H(G)$ is the distribution character of $\rho$, normalized as a measure so that $\pi(\Theta_{\rho})$ is the projection onto $\check\rho$-isotypic components, where $\check\rho$ is the dual representation.
\item  An invariant distribution $f\in D(G)^G$ 
     is called 
$U$-\emph{admissible} with respect to a compact open pro-$p$ subgroup $U\subset G$ if  
for any compact open $U'\subset U$
 and $(\rho,V_\rho)\in \Irr(U')$, we have 
\[f*\Theta_\rho\neq0\Rightarrow \exists g\in G\text{ such that } (V_\rho)^{U'\cap gUg^{-1}}\neq 0.\]

\end{enumerate}
\end{definition}

    \begin{lemma}\label{projector}
Let $z\in Z(G)$, and let $(\rho,V_\rho)\in\Irr(K)$ be an irreducible representation of a compact open pro-$p$ subgroup $K\subset G$. Then
\[
f_z*\Theta_\rho\neq 0
\]
implies that there exists $\pi\in R(G)$ such that
\[
z(\pi)\neq 0
\qquad\text{and}\qquad
\operatorname{Hom}_K(\check\rho,\pi)\neq 0.
\]
%Here $\check\rho$ is the dual of $\rho$.
Moreover, if
\[
R(G)=\mathcal R(G)^{\heartsuit}\oplus
 R(G)^{\spadesuit}
\]
is a decomposition such that $z\in Z(G)^{\heartsuit}$, then $\pi$ may be chosen in $ R(G)^{\heartsuit}$.
\end{lemma}
\begin{proof}
Note that  $f_z*\Theta_\rho\neq0$
if and only if 
$\pi(f_z*\Theta_\rho)\neq0
\text{\ for some\ } \pi\in R(G)$.
Since \[\pi(f_z*\Theta_\rho)=z(\pi)\circ\pi(\Theta_\rho),\] we see that 
$f_z*\Theta_\rho\neq0$
implies  
 $z(\pi)\neq 0\text{ and } \pi(\Theta_\rho)\neq 0
\text{ for some } \pi\in R(G)$,
%On the other hand, \[\frac{\dim(\rho)}{\on{vol}(K)}\pi(\Theta_\rho):V_\pi\to V_\pi\] is the projection onto the $\check\rho$-isotypic component of $V_\pi$. Thus we conclude that $f_z*\Theta_\rho\neq0$ implies 
 i.e. $z(\pi)\neq 0\text{ and } \operatorname{Hom}_K(\check\rho,\pi)\neq 0 \text{\ for some\ } \pi\in R(G)$, which proves the first assertion.

Now assume $z\in Z(G)^\heartsuit$.
Then 
 we have 
\[\pi(f_z*\Theta_\rho)=\pi^\heartsuit(f_z*\Theta_\rho)\oplus\pi^\spadesuit(f_z*\Theta_\rho)\] 
where $\pi=\pi^\heartsuit\oplus\pi^\spadesuit\in R(G)=R(G)^\heartsuit\oplus R(G)^\spadesuit$.
Since  $z(\pi^\spadesuit)=0$, we have  $\pi^\spadesuit(f_z*\Theta_\rho)=z(\pi^\spadesuit)\circ\pi^\spadesuit(\Theta_\rho)=0$ and it follows that 
\[\pi^{\heartsuit}(f_z*\Theta_\rho)=\pi(f_z*\Theta_\rho)\]
and we can replace $\pi$ by the summand $\pi^\heartsuit\in R(G)^\heartsuit$
in the argument above. 
\end{proof}

Here is our key observation:

\begin{thm}\label{main}
Suppose that $U\subset G$ is a compact open pro-$p$ subgroup and $R(G)=R(G)^{\heartsuit}\oplus R(G)^{\spadesuit}$ is a
 decomposition
%with the property that there exists  a compact open pro-$p$ subgroup $U\subset G$ such that 
such that any $(\pi,V_\pi)\in R(G)^\heartsuit$ is generated by its $U$-fixed vectors $(V_\pi)^U$. Then for any $z\in Z(G)^\heartsuit$,  the corresponding 
invariant distribution $f_z$ is $U$-admissible.
 
\end{thm}
\begin{proof}
Let $\rho\in\Irr(U')$ be an irreducible representation of a compact open subgroup $U'\subset U$
such that  $f_z*\Theta_\rho\neq0$.
By Lemma \ref{projector}, there exists 
 $(\pi,V_\pi)\in R(G)^\heartsuit$ 
 satisfying $z(\pi)\neq0$ and 
 $\Hom_{U'}(\check\rho,\pi)\neq 0$.
Choose a non-zero $U'$-equivariant embedding 
$f:V_{\check\rho}\hookrightarrow V_\pi$ and let 
$0\neq w$ lie in its image. 
Since $V_\pi$ is generated by its $U$-fixed vectors, we can 
write
$w=\sum_{j=1}^n g_jv_j$ for some
$v_j\in (V_\pi)^{U}$ and $g_j\in G$. Note that we have 
$g_jv_j\in V_{\pi}^{U'\cap g_jUg_j^{-1}}$ for 
any compact open $U'\subset U$ and 
$j=1,...,n$. On the other hand, we have
\[w=\pi(\Theta_\rho)(w)=\sum_{j=1}^n\pi(\Theta_\rho)(g_jv_j)\neq 0\]
and it follows that 
\[\pi(\Theta_\rho)(g_iv_i)\neq 0\]
for some $i\in [1,n]$.
Since  $\pi(\Theta_\rho):V_\pi\to \Hom_{U'}(\check\rho,\pi)\otimes V_{\check\rho}\subset V_\pi$
is $U'$-equivariant and 
$g_iv_i$ is fixed by $U'\cap g_iUg_i^{-1}$ we conclude that 
\[0\neq\pi(\Theta_\rho)(g_iv_i)\in\Hom_{U'}(\check\rho,\pi)\otimes (V_{\check\rho})^{U'\cap g_i Ug^{-1}}.\]
Thus $(V_{\check\rho})^{U'\cap g_i Ug^{-1}}\neq0$
and, since $U'\cap g_i Ug^{-1}$ is compact
 pro-$p$ and $\ell\neq p$, it implies that 
$(V_\rho)^{U'\cap g_i Ug^{-1}}\neq0$.
\end{proof}

\section{Reductive groups}\label{reductive}
In this section we assume $G$ is a connected reductive 
group over a $p$-adic field $F$ where $\operatorname{char}(F)=0$.
%We  assume $\ell=0$ or is sufficiently large compared to $p$.

\subsection{Local integrability }
Recall the following important results of Harish-Chandra \cite{HC99}, and their generalization to mod $\ell$ coefficients in \cite{T26}:

\begin{thm}\label{HC99}
Let $f$ be an admissible invariant distribution on $G$. Then $f$ satisfies the following
\begin{enumerate}
    \item $f$ admits local character expansions as in \cite[Theorem 16.2]{HC99}.
    \item the distribution $f|_{\Grs}$
    is represented by a locally constant function 
    $F:\Grs\to k$ on the regular semisimple locus $\Grs\subset G$.
    \item When $k=\mathbb{C}$, $F$ is locally-$L^1$ on $G$ 
    and $f$ is represented by $F$. 
    \item
    When $\ell=0$ or $\ell$ is sufficiently large in the sense of \cite[Remark 3.2]{T26}, $f$ can be represented by $F$ in the sense that it can be constructed via iterated geometric series {\it loc. cit.}.
\end{enumerate}
\end{thm}
\begin{proof}
    When $k=\mathbb C$, part (1), (2) and (3) 
    are  \cite[Theorem 16.1, 16.2, 16.3]{HC99}.
    In general, as observed in \cite[Th\'{e}ore\`{m}e 5.4.4]{VW01},  the same argument in \cite{HC99} implies part (1). Part (2) is always a special case of (1) when we consider local character expansions near regular semisimple elements. Part (4) follows from Part (1) and the main result of \cite{T26}.
\end{proof}

\begin{corollary}\label{local int for z} Let $Z(G)^{\heartsuit}$ be as in Theorem \ref{main}. For any
$z\in Z(G)^\heartsuit$, the corresponding 
invariant distribution $f_z$ satisfies 
(1)-(4) of Theorem \ref{HC99}.
\end{corollary}
\begin{proof}
    Since $U$-admissibility implies admissibility, the claim follows from Theorem \ref{main} and Theorem \ref{HC99}.
\end{proof}

\subsection{Depths} For the rest of this section we assume $G$ splits over a tamely ramified extension. According to \cite{MP94,MP96,BD84} if $\ell=0$ and  \cite{V96} and \cite[Appendice A]{D09} in general, %if $\ell>0$,
for any non-negative rational number $r\in\mathbb Q_{\geq0}$, we have the depth decomposition $R(G)=R(G)_{\leq r}\oplus R(G)_{>r}$
where $R(G)_{\leq r}$
(resp. $R(G)_{>r}$)
consists of smooth representations with
irreducible subquotients of depths $\leq r$ (resp. $>r$). Let 
$K_r:=\cap_{x\in\overline{C}}G_{x,r+}$ where $\overline{C}$ is the closure of a fixed alcove.

\begin{lemma}\label{K_r} 
 The depth decomposition $R(G)=R(G)^{\heartsuit}\oplus R(G)^{\spadesuit}$, with 
   $R(G)^{\heartsuit}=R(G)_{\leq r}$
   and $R(G)^{\spadesuit}= R(G)_{>r}$, %is nice in the sense of Definition \ref{nice}, that is, there exists 
   is such that any $\pi\in R(G)^{\heartsuit}$ is generated by $K_r$-fixed vectors.
   %a compact open pro-$p$ subgroup $U\subset G$ such that    any $(\pi,V_\pi)\in R(G)_{\leq r}$ is generated by $U$-fixed vectors.
\end{lemma}
\begin{proof}
It follows from 
\cite{MP94} if $\ell=0$ and 
\cite[Lemme A.3]{D09} if $\ell>0$ that  
any $(\pi,V_\pi)\in R(G)_{\leq r}$ is generated 
by $K_r$-fixed vectors; indeed, the progenerators $P(r)$ in \cite[p. 330]{D09} are generated by $K_r$-fixed vectors by their very construction. 
%Since the family of compact open pro-$p$ subgroups of $G$ is cofinal in the set of open compact subgroups we can take $U\subset K_r$ to be any such pro-$p$ subgroup.  
\end{proof}

Consider the depth decomposition $Z(G)=Z(G)_{\leq r}\oplus Z(G)_{>r}$ of the Bernstein center
where 
$Z(G)_{\leq r}=\{z\in Z(G)|z|_{R(G)_{>r}}=0 \}$
and $Z(G)_{>r}=\{z\in Z(G)|z|_{R(G)_{\leq r}}=0 \}$.

\begin{corollary}\label{bounded depth}
      For any $z\in Z(G)_{\leq r}$
the corresponding invariant distribution 
$f_z$ satisfies 
(1)-(4) of Theorem \ref{HC99}.
\end{corollary}
\begin{proof}
It follows from Corollary \ref{local int for z} and Lemma \ref{K_r}.
\end{proof}

\subsection{Bernstein supports}
According to \cite{BD84} if $\ell=0$ and \cite[Theorem 4.22]{DHKM24} if $\ell$ is banal, there is a Bernstein decomposition $R(G)=\prod_{[M,\sigma]} R(G)_{[M,\sigma]}$ into the so-called Bernstein blocks $R(G)_{[M,\sigma]}$, where the index set $[M,\sigma]$ is the inertial equivalence classes of supercuspidal supports. We say that $z\in Z(G)$ has bounded Bernstein support if its restriction to all but finitely many Bernstein blocks is zero.

\begin{corollary}\label{supp}\label{bounded support} Suppose $\ell=0$ or $\ell$ is banal. For any $z\in Z(G)$ with bounded Bernstein support, 
the corresponding invariant distribution 
$f_z$ satisfies 
(1)-(4) of Theorem \ref{HC99}.
\end{corollary}
\begin{proof}
Since unramified twists  and parabolic induction preserve depth,
we have 
$R(G)_{[M,\sigma]}\subset R(G)_{\leq r}$ for some $r$.
Thus any $z\in Z(G)$ with bounded Bernstein support lies in $Z(G)_{\leq r}$ for some sufficiently large $r$ and the claim follows from Corollary \ref{bounded depth}.
\end{proof}
\begin{remark}
    In the case $\ell=0$, Corollary \ref{bounded depth} or \ref{supp} above provides an alternative argument for the main results in \cite{MT02}. 
    The case $\ell>0$ appears to be new and does not follow directly from the methods of \emph{loc. cit.}
\end{remark}

\begin{remark}
Consider the case where $G$ is a reductive group over a local function field $F$ of characteristic $p>0$. We expect Harish-Chandra's Theorem (Theorem \ref{HC99}(1)) to remain valid when $p$ is sufficiently large. When this is the case, all results in Section \ref{reductive} follow. 
Moreover, with some modifications, we expect that the results of the note remain valid when $k$ is replaced by any $\mathbb Z[1/p]$-algebra.

\end{remark}

\subsection*{Acknowledgments} We thank the NCTS-National Center for Theoretical Sciences at Taipei where parts of this work were done. The research of T.-H. Chen is supported by NSF grant DMS-214372
and Simons Fellowships.
The research of C.-C. Tsai is supported by NSTC grant 115-2628-M-001-002.

\end{document}